\RequirePackage{fix-cm}
\documentclass[smallextended]{svjour3}       
\smartqed  
\usepackage{amsmath}
\usepackage{amssymb}
\usepackage{tabularx}
\usepackage{comment}
\usepackage{relsize}
\usepackage{enumerate}
\usepackage{tikz}
\usetikzlibrary{positioning}
\usepackage{tikz-cd}
\usepackage{color}
\usepackage{tabularx}
\usepackage{ragged2e}
\usepackage{enumitem}
\newcolumntype{C}{>{\Centering\arraybackslash}X}
\DeclareMathAlphabet\mathfrak{U}{euf}{m}{n}
\SetMathAlphabet\mathfrak{bold}{U}{euf}{b}{n}
\usepackage[rightcaption]{sidecap}
\sidecaptionvpos{figure}{c}
\usepackage{scrextend}
\usepackage{centernot}
\newcommand{\Z}{\mathbb{Z}}
\newcommand{\F}{\mathbb{F}}

\newcommand{\N}{\mathbb{N}}

\newcommand{\mbb}[1]{\mathbb{#1}}

\newcommand{\mcal}[1]{\mathcal{#1}}
\newcommand{\txt}[1]{\textrm{#1}}
\newcommand{\cb}[1]{ \big \{ #1 \big \} }
\newcommand{\ncb}[1]{ \left \{ #1 \right \} }
\newcommand{\cp}[1]{ \big ( #1 \big ) }

\newcommand{\es}{\enskip}

\newcommand{\ba}[2]{\left [\begin{array}{#1} #2 \end{array}\right ]}

\newcommand\restr[2]{{\left.\kern-\nulldelimiterspace#1 \vphantom{\big|}\right|_{{#2}}}}

\begin{document}

\title{Ramsey-Theoretic Characterizations of Classically Non-Ramseyian Problems
}


\author{
Bryce Alan Christopherson         
}


\institute{B. Christopherson \at
              Witmer Hall Room 313\\
              101 Cornell Street Stop 8376\\
              Grand Forks ND 58202-8376 \\
              Tel.: +1-701-777-2881\\
              \email{bryce.christopherson@UND.edu}
}

\date{}

\maketitle

\begin{abstract}
    In this paper, we will develop a \textit{significantly} more general notion of classical Ramsey numbers (extending most other graph-theoretic generalizations) and make some preliminary characterizations of these new Ramsey numbers using simple algebraic tools.  Throughout, we make a case arguing that, while our access to {\em specific values} of Ramsey numbers (or, in general, precise numerical solutions to Ramsey-theoretic problems) may be limited, the {\em interplay between} and {\em overall structure of} Ramseyian objects is likely tractable.  To support the relevancy of this perspective, we conclude by demonstrating that the Green-Tao Theorem, the Twin Prime conjecture, Zhang's bounded prime gap theorem, and Polignac's conjecture can be viewed as statements about Ramsey numbers.
    
\end{abstract}

\section{Introduction}

Ramsey theory addresses the mathematical phenomenon which seems to dictate that the {\em local} structure of certain ordered collections of objects often necessarily becomes more regular beyond a certain bifurcation point, provided there is a suitable degree of {\em global} structure present in the ordered collection of objects themselves.  Usually, a generic Ramsey-theoretic problem seeks to determine the precise value at which this bifurcation occurs.  There are many notions of Ramsey properties (see, e.g., \cite{gould2010ramsey,graham2008old,graham2007some}).  The most well known are the classical Ramsey numbers.

\begin{definition}[Classical Ramsey Numbers]
    Let $K_n$ denote the undirected complete graph on $n$ vertices, and let $z_1,\hdots,z_m$ be positive integers.  The \textit{Ramsey number} $R(z_1,\hdots,z_{m})$ is  the smallest $n$ such that, for any edge coloring of $K_n$ with $m$ colors, there exists a subgraph of $K_n$ isomorphic to $K_{z_i}$ that is monochromatic in the $i^{th}$ color for some $i \in \cb{1,\hdots,m}$.
\end{definition}


There are a number of different notions extending and generalizing classical Ramsey numbers \cite{burr1974generalized,harary2006recent,chvatal1972generalized,conlon2015erdHos,burr1980extremal}.  In this paper, a more general (perhaps too general) extension related to what was initially proposed in \cite{christopherson} will be considered that covers if not all, then most all, variants that have been considered. 

 This generalization, though moderately cumbersome to work with and perhaps not ideal in its technical formulation, does offer the useful ability to `see everything at once' in a single context.  Instead of proposing this generalization as a good way to approach these problems, its purpose is instead to allow us a better vantage to take stock of them, see their essential structure, and gain some insight as to what a `correct' overarching formulation might be.  
 
 We will approach things using some basic tools to convert the general problem into a digestible algebraic one that more clearly allows us to ascertain the degree and nature of the difficulties that can arise.  The approach, started in \cite{christopherson} and similar to later approaches such as \cite{dealgebraic}, is basically the following: 
 
 \begin{enumerate}
    \item Sets $S_{i}$ of admissible edge-colorings of the base graphs $G_{i}$ are associated to a set of points $\mcal{O}_i$.
    \item  The base graphs $G_i$ and target graphs $\mbb{X}_i=\left \{ X_j\right \}_{{j \in J_i}}$ with their target edge-colorings $C_i = \left \{\psi_j\right \}_{{j\in J_i}}$ are used to generate an ideal $I_i$ in a polynomial ring.
    \item  The members of $I_i$ vanish on a set of points corresponding to edge-colorings of $G_i$ for which there is a restriction to a subgraph isomorphic to $X_j$ whose coloring coincides with $\psi_j$ for some $j \in J_i$.
    \item The admissible edge-colorings of $G_i$ restrict to target colorings on subgraphs isomorphic to target graphs if and only if the algebraic set $V(I_i)$ of points on which all members of $I_i$ vanish contains $\mcal{O}_i$.
    \item The element $n\in I$ (if such a value exists) of least order at which $V(I_t)\subseteq \mcal{O}_t$ for all $n \leq t \in I$ coincides with the value of the Ramsey number $R_{(\mbb{G},\mcal{S})}(\mbb{X},\mcal{C}) = n$.
\end{enumerate}

The structure of this paper can be summarized as follows:  In Section~\ref{sec1}, we begin by giving an alternate, equivalent formulation of classical Ramsey numbers and a definition of the new version we introduce, as well as some preliminary tools and results to build familiarity with these objects.  In Section~\ref{sec2}, we focus on a particular type of Ramsey numbers and highlight a straightforward connection between these numbers and the zeros of certain polynomials.  Using this connection, we are able to provide some helpful characterizations.  In Section~\ref{sec3}, we employ a small trick so we can apply the results of the previous sections a much larger class of Ramsey numbers.  In Section~\ref{sec4}, we give an example of a few basic `structural' results for Ramsey numbers.  In Section~\ref{sec6}, we provide some alternative, Ramsey-theoretic formulations for a few well-known problems from other mathematical fields.  In Section~\ref{sec7}, we make some closing remarks and summarize.

For the purposes of brevity, we will frequently make use of standard graph-theoretic, algebraic, and algebro-geometric terms and results without much introductory explanation throughout the paper -- for common references, see, e.g., \cite{godsil2001algebraic,bollobás1998modern,harris2013algebraic,eisenbud1995commutative,kemper2010course,isaacs2009algebra}.  Given a prime power $q=p^k$, we will adopt the convention of denoting the finite field of order $q$ by $\F_q$. 

\section{Preliminary Work}\label{sec1}

In this section, we will develop an alternative formulation of classical Ramsey numbers, as well as a generalization of classical Ramsey numbers which emerges naturally in this context.  To begin, we define a graph coloring.

\begin{definition}[Graph Coloring]\label{graph coloring} \ 
    Let $G$ be a finite, undirected graph and let $E(G)$ denote the edge set of $G$.  Given a set $A$ of size $m$, we say a function $\rho:E(G)\rightarrow A$ is an {\em $A$-valued $m$-edge-coloring of $G$} and denote the set of all such $A$-valued $m$-edge-colorings of $G$ by $\textrm{Hom}(G,A)$.  For $a \in A$, we denote by $a_G \in \textrm{Hom}(G,A)$ the constant mapping $a_G(e)=a$ for all $e\in E(G)$ and call such a coloring \textit{monochromatic}.
\end{definition}

We will require the following definition which will be used heavily throughout the paper.  If we say a graph $G$ is isomorphic to a graph $H$, we mean there exists a map $\pi:V(G)\rightarrow V(H)$, called a {\em graph isomorphism}, which satisfies $\pi\cp{V(G)}=V(H)$ and preserves the incidence structure of $G$ -- meaning that the vertices $v, u \in V(G)$ are (resp. are not) adjacent in $G$ if and only if the vertices $\pi(v),\pi(u) \in V(H)$ are (resp. are not) adjacent in $H$.  Since we will concern ourselves primarily with the edge set of a graph, given a graph isomorphism $\pi$, let $\pi_*:E(G)\rightarrow E(H)$ denote the map induced by $\pi$ which sends the edge $e \in E(G)$ incident to vertices $v$ and $w$ to the edge $\pi_*(e)\in E(H)$ incident to $\pi(v)$ and $\pi(w)$.  

\begin{definition}[Isomorphism Set] \
    Let $G$ and $X$ be finite, undirected graphs.  Write $Y \leq G$ if $Y$ is a subgraph of $G$ and write $Y \cong_\pi X$ if $\pi:V(Y)\rightarrow V(X)$ is a graph isomorphism.   Then, we say the {\em isomorphism set of $X$ in $G$} is the set $G/X = \cb{\pi_*: Y\leq G, \es Y \cong_\pi X}$.
\end{definition}

The isomorphism set of a graph $X$ in a graph $G$ will serve as a useful way to write formulas without making an explicit reference to the subgraphs of $G$ or any assumptions as to whether or not $X$ is isomorphic to {\em any} subgraphs of $G$, since $G/X$ is still defined in such cases (that is, $G/X=\emptyset$ if $X$ is not isomorphic to a subgraph of $G$).  We will now provide some examples to familiarize the reader with this notion.

\begin{example} 
    Let $G$ be a finite, undirected graph.  Then, $G/G = \cb{\pi_*: Y \leq G, Y \cong_{\pi} G}$ by definition. So, $G/G= \cb{\pi_* : \pi\in \txt{Aut}(G)}$. For a less trivial example, the number of elements belonging to $K_4/K_3$ can be determined by noticing that there are $\binom{4}{3}=4$ distinct subsets of $3$ vertices of $K_4$ and $|S_3|=3!=6$ automorphisms of each subgraph induced by such a subset.  So, $|K_4/K_3|=4!=24$.
\end{example}

For the remainder of the paper, we will do away with the asterisk subscript and simply write $\pi \in G/X$.  Using the two definitions given above, we are now ready to provide a somewhat more symbolically nuanced (though, ultimately, equivalent) reformulation of classical Ramsey numbers. 

\begin{definition}[Ramsey Numbers]\label{RN} \ 
    Let $A$ be an $m$-set and let $\cb{z_j}_{j\in A} \subseteq \Z_{>0}$.  The \textit{Ramsey number} $R(z_a: a \in A)$ is  the smallest $n$ such that for any $t \geq n$ and any $\rho \in \textrm{Hom}(K_t,A)$, there exists some $a \in A$ and $\pi \in K_t/K_{z_a}$ satisfying $\restr{\rho}{\pi^{-1}(K_{z_a})}= a_{K_{z_a}}\circ \pi$.
\end{definition}

Of course, it is natural to ask when or if it is possible to guarantee that some suitably ordered sequence of sufficiently 'nice' graphs eventually possesses a similar property.  Namely, given two index sets $I$ and $J$ with some sufficiently amicable ordering and two sequences of graph $\cb{G_i}_{i \in I}$ and $\cb{X_j}_{j \in J}$, when is it the case that there exists $n \in I$ such that, for any $k \geq n$, any {\em global} $m$-edge-coloring of $G_k$ has a specified (possibly non-constant) {\em local} restriction to a subgraph of $G_k$ isomorphic to $X_j$ for some $j \in J$?

Naturally, for the question to be answered in the affirmative, it seems apparent that at least one of the graphs belonging to $\cb{X_j}_{j\in J}$ must be isomorphic to a subgraph of $G_i$ for all indices $i\in I$ beyond a certain point.  So, for this reason, we give a definition that will simplify things considerably.  In what follows, we will use $I$ to denote our index set and will assume it has a strict well-ordering--that is, $I$ has a strict total order and every nonempty subset of $I$ has a unique minimal element with regard to the restriction of the order to that subset. Given such a strictly well-ordered set $I$, we will let let $0$ denote the unique minimal element of $I$, and let $k$  denote the unique minimal element of $I\setminus \cb{t \in I: 0 \leq t < k}$ for appropriately sized $k$, writing $I_{\geq k}=\cb{t \in I : k \leq t}$.  That is, $I_{\geq k}$ is the subset of $I$ obtained by removing the $k$ elements of lowest order from $I$. As a shorthand, if $T$ is a set indexed by $I$, we write $T_{\geq k}$ to denote the subset of $T$ with indices in $I_{\geq k}$.

\begin{definition}[Constituent, Hereditary] \ 
    Let $I$ be a strictly well-ordered set and let $\mbb{G}=\cb{G_i}_{i \in I}$ and $\mbb{X}=\cb{X_j}_{j \in J}$ be collections of finite, undirected graphs. Then, if $G_i/X_j$ is nonempty for all $i \in I$ and all $j \in J$, we say $\mbb{X}$ is {\em constituent} in $\mbb{G}$ and write $\mbb{X}\unlhd \mbb{G}$. If $\cp{\mbb{G} \setminus \mbb{G}_{{\geq i}}} \unlhd \mbb{G}_{\geq k}$ for every $0< i \leq k$, we say that $\mbb{G}$ is {\em hereditary}.
\end{definition}

Broadly speaking, a set of graphs $\mbb{G}$ indexed by a strict well-ordered set $I$ is hereditary if the graphs belonging to $\mbb{G}$ induce an order via containment that agrees with the order on $I$. That is, heredity is a {\em global} property belonging to $\mbb{G}$.  To contrast, a set of graphs $\mbb{X}$ is constituent in a set of graphs $\mbb{G}$ indexed by strict well-ordered set if {\em each} element of $\mbb{G}$ contains a subgraph isomorphic to {\em each} element of $\mbb{X}$.  That is, constituency is a {\em local} property of the elements of $\mbb{G}$ relative to the set $\mbb{X}$.  It is also worth mentioning that, in general, we will not require the entirety of a set $\mbb{X}$ to be constituent in $\mbb{G}$ in most situations where such a pair will arise.  Since it is easier to do so, however, we present their definitions together for the sake of clarity.  

We introduce two final definitions before proceeding to our generalization of Ramsey numbers.

\begin{definition}[Ramsey base]
    Let $\mbb{G}=\cb{G_i}_{i\in I}$ be a collection of undirected graphs and let $\mcal{S}=\cb{S_i}_{i \in I}$ be a collection of subsets $S_i \subseteq \textrm{Hom}(G_i,A_i)$ for some sets $A_i$.  If $\mbb{G}$ is hereditary, we say the pair $(\mbb{G},\mcal{S})$ is a {\em Ramsey base}. 
\end{definition}

Ramsey bases will serve as the 'basic' graphs we color and check for certain subcolorings.  The counterpart to this consists of what we're checking for. 

\begin{definition}[Ramsey Symbol]
    Let $\mbb{X}=\cb{\mbb{X}_i}_{i\in I}$ be a collection of sets $\mbb{X}_i=\cb{X_j}_{{j\in J_i}}$ of finite, undirected graphs indexed by a finite set $J_i$ for each $i \in I$, and let $\mcal{C}=\cb{C_i}_{i\in I}$ be a collection of sets $C_i=\cb{\psi_j}_{{j\in J_i}}$ of $A_i$-valued edge colorings $\psi_j\in \textrm{Hom}(X_j,A_i)$ for all $j \in J_i$.  If $\cb{\coprod_{{j\in J_i}} X_j}_{i\in I}$ is hereditary, we say the set $(\mbb{X},\mcal{C})$ is a {\em Ramsey symbol}.  
\end{definition}

Broadly, viewing a Ramsey number with a given Ramsey base as a function, Ramsey symbols will serve as the {\em argument} of these functions.  For the remainder of the paper, given a Ramsey symbol $\cp{\mbb{X},\mcal{C}}$, we will implicitly take $I$ as the index set of $\mbb{X}$ and $\mcal{C}$, and we will likewise take $J_i$ as the index set of $\mbb{X}_i$ and $C_i$.  With these definitions in hand, we now have the framework to produce the generalization of Ramsey numbers previously described.

\begin{definition}[Ramsey Numbers]\label{GRN} \ 
    Let $(\mbb{G},\mcal{S})$ be a Ramsey base and let $(\mbb{X},\mcal{C})$ be a Ramsey symbol.  The \textit{Ramsey number} $R_{(\mbb{G},\mcal{S})}\left(\mbb{X},\mcal{C}\right)$ is the smallest $n \in I$, if such an $n$ exists, such that, for any $t \geq n$ and any $A_t$-valued edge-coloring $\rho \in S_t \subseteq \textrm{Hom}(G_t,A_t)$, there exists some $j \in J_t$ and $\pi \in G_t/X_j$ satisfying $\restr{\rho}{\pi^{-1}(X_j)} = \psi_j \circ \pi$.  If such an $n$ does not exist, then we say the Ramsey number $R_{(\mbb{G},\mcal{S})}\left(\mbb{X},\mcal{C}\right)$ does not exist.
\end{definition}

The above definition is given in what will be, for practical purposes, too much generality.  Throughout a substantial portion of the body of this paper, this generality will exceed the scope of what is useful.  All the same, it is useful to state in this overly general form, since it more readily highlights the restrictions that need to be imposed to reach cases that can be dealt with.  That is, we will frequently require our Ramsey base and Ramsey symbol to have certain desirable properties in order to make conclusions about the existence or nonexistence of a corresponding Ramsey numbers.  After proving some results for these more restrictive cases in early sections, we will return to a more general case closer to the definition above by using certain decompositions and embeddings to convert our previously obtained results to this broader context.  The following definitions make precise the nature of the constraints we will initially require.


\begin{definition}[Locally Finite, Finite, Maximal] 
    Let $(\mbb{G},\mcal{S})$ be a Ramsey base, where $\mathbb{G}=\cb{G_i}_{i \in I}$ and $\mcal{S}=\cb{S_i}_{i\in I}$ with $S_i \subseteq \textrm{Hom}(G_i,A_i)$. Then:
    \begin{enumerate}
        \item If $|A_i|< \infty$ for all $i \in I$, we say it is {\em of locally finite type}.
        \begin{enumerate}[label=\theenumi.\alph*.]
            \item If, additionally, $A_i=A_j$ and all $i,j \in I$, we say it is {\em of finite type}.  
        \end{enumerate}
        \item If $S_i = \textrm{Hom}(G_i,A_i)$ for all $i \in I$, we say it is {\em maximal}.
    \end{enumerate}
\end{definition}


\begin{definition}[Locally Finite, Finite, Uniform] Let $(\mbb{X},\mcal{C})$ be a Ramsey symbol, where $\mathbb{X}=\cb{X_i}_{i \in I}$, $X_i = \cb{X_j}_{j \in J_i}$, and $\mcal{C}=\cb{C_i}_{i\in I}$ with $C_i = \cb{\psi_j}_{j \in J_i}$ for some $\psi_j \in \textrm{Hom}(X_j,A_i)$.  Then:
 
\begin{enumerate}
\item If $|A_i|< \infty$ for all $i\in I$, we say it is {\em of locally finite type}.  
\begin{enumerate}[label=\theenumi.\alph*.]
    \item If, additionally, $A_i=A_j$ for all $i,j \in I$, we say it is {\em of finite type}. 
\end{enumerate}
\item If $A_i \subseteq A_t$ and $(\mbb{X}_i,C_i)=(\mbb{X}_t,\mcal{C}_t)$ for all $t \geq i$ (where the equality is understood as factoring through an inclusion if $A_i \neq A_t$), we say it is {\em uniform}.
\end{enumerate}
\end{definition}


As a natural shorthand, if the Ramsey base $(\mbb{G},\mcal{S})$ is maximal, we write $R_{\mbb{G}}\left(\mbb{X},\mcal{C}\right)$ instead of $R_{(\mbb{G},\mcal{S})}\left(\mbb{X},\mcal{C}\right)$ for simplicity.  Similarly, if the Ramsey symbol $(\mbb{X},\mcal{C})$ is uniform, we will omit reference to the index set $I$ for $\mbb{X}$ and $\mcal{C}$ and write $\mbb{X}=\left \{X_j \right \}_{j\in J}$ and $\mcal{C}=\left \{\psi_j \right \}_{j\in J}$ instead. 

Notice the minor asymmetry between the constraints placed on Ramsey symbols and Ramsey bases in item $(2)$ of each of the above definitions.  This is mainly due to the fact that, while we {\em could} certainly define maximal Ramsey symbols or uniform Ramsey bases in the same fashion, it is straightforward to see that they would be (in general) of relatively little interest, since maximal Ramsey symbols would, essentially, measure only subgraph inclusion and uniform Ramsey bases would allow for only a trivial form of heredity.

Using the above definitions, notice that setting $\mbb{G}=\ncb{K_n}_{n\in \N}$ and letting $\mbb{X} = \ncb{K_{z_j}}_{{j\in [m]}}$, $\mcal{C}=\ncb{j_{K_{z_j}}}_{{j\in [m]}}$ with $[m] = \ncb{1,\hdots,m}$ yields $R_{\mbb{G}}\left(\mbb{X},\mcal{C} \right) = R\left (z_1,\hdots,z_{m}\right)$.  So, the Ramsey numbers we describe are a generalization of classical Ramsey numbers.  Likewise, the generalized Ramsey numbers $R(H_1,\hdots,H_m)$ for graphs $H_1,\hdots,H_m$--as described in; e.g.,  \cite{burr1974generalized}--can be produced from our Ramsey numbers in the same way by simply replacing $\mathbb{X}=\ncb{H_j}_{{j\in [m]}}$ and $\mathcal{C}=\ncb{j_{H_j}}_{j \in [m]}$.  

Indeed, as a Ramsey number of the form we describe in Definition~\ref{GRN}, both the classical and generalized Ramsey numbers previously studied are of a particularly convenient form: they have maximal bases of finite type and a uniform symbols of finite type.  Let's look at some further examples to familiarize ourselves with these constraints.

\begin{example}
    Let $\mbb{G}=\ncb{K_i}_{i \in \N}$ and suppose $\mcal{S}=\ncb{S_i}_{i\in \N}$ is such that for any $t \in \N$, the set $S_t$ consists of the $A_t$-valued edge colorings of $G_t$, where $|A_t|=t(t-1)/2$, such that no two edges of $G_t$ are assigned the same element in $A_t$.  Then, since $\mbb{G}$ is hereditary and $|S_i| < \infty$ for all $i \in I$, $(\mbb{G},\mcal{S})$ is a Ramsey base of locally finite type.  However, since $S_i \neq S_t$ for $i \neq t$, the Ramsey base $(\mbb{G},\mcal{S})$ is {\em not} of finite type.
\end{example}
\begin{example}
    Let $\mbb{X}=\ncb{\mbb{X}_i}_{i\in \N}$, where $\mbb{X}_i = \ncb{K_{2j}}{j\in \N}$ for all $i \in \N$, and let $\mcal{C}=\ncb{C_i}_{i\in \N}$, where $C_i = \ncb{\textrm{Hom}(K_{2j},A)}_{j\in \N}$ for some set $A$ such that $|A|<\infty$ and for all $i \in \N$.  Then, $(\mbb{X},\mcal{C})$ is {\em not} a Ramsey symbol, since the index set $J_i = \N$ of $\mbb{X}_i$ is not finite.  However, if $(\mbb{X},\mcal{C})$ {\em was} a Ramsey symbol, then as $(\mbb{X}_i,C_i)=(\mbb{X}_k,C_k)$ for all $i,k \in \N$, $(\mbb{X},\mcal{C})$ it would have been a uniform Ramsey symbol, and as both $A_i=A$ for all $i \in I$ and $|A|<\infty$, $(\mbb{X},\mcal{C})$ would also have been a Ramsey symbol of finite type.
\end{example}

As mentioned previously, certain cases of Ramsey numbers will prove to be more or less amicable to study.  Among those cases, two particular ones are worth distinguishing with names.

\begin{definition}[Exact, Galois Type]
    Let $(\mbb{G},\mcal{S})$ be a Ramsey base and let $(\mbb{X},\mcal{C})$ be a uniform Ramsey symbol.  If $\mbb{V} \unlhd \mbb{G}$ for some nonempty subset $\mbb{V} \subseteq \mbb{X}$, we say that the Ramsey number $R_{(\mbb{G},\mcal{S})}\left(\mbb{X},\mcal{C}\right)$ is \textit{exact}.  
    
    \noindent If the Ramsey base $(\mbb{G},\mcal{S})$ is of finite type, the Ramsey symbol $(\mbb{X},\mcal{C})$ is uniform and of finite type, and $|A|=p^k$ for a prime $p$ and some positive integer $k$, we say that the Ramsey number $R_{(\mbb{G},\mcal{S})}\left(\mbb{X},\mcal{C}\right)$ is \textit{of Galois type}.
\end{definition}

For an example of a classical Ramsey number of Galois type, consider any one where the number of colors used is a prime power.  We will address this case more later.  

To start, however, the remainder of this section is dedicated to illustrating the general approach we will employ, which is to show that complicated Ramsey numbers can be converted into simpler ones.   We use the exact Ramsey numbers for this, since they are particularly simple.  We will do this by building relatively simple `tools' to relate Ramsey numbers with different bases and symbols.  The first example of such a tool is very straightforward and will allow us to remove unnecessary portions of a Ramsey base.

\begin{definition}[Resolutions]\label{exact xres}
    Let $\mbb{G}=\cb{G_i}_{i\in I}$ be hereditary and let $\mbb{X}=\cb{X_j}_{j \in J}$ be a collection of finite, undirected graphs. We say {\em $\mbb{G}$ admits a $\mbb{X}$-resolution} if there is a minimum value $k$ such that $\mathbb{V} \unlhd \mathbb{G}_{\geq k}$ for some $\mathbb{X}\supseteq \mathbb{V} \neq \emptyset$.  If $\mbb{G}$ admits an $\mbb{X}$-resolution, we say $\mbb{G}_{\geq k}$ is the {\em $\mbb{X}$-resolution of $\mbb{G}$} and write $\txt{res}_\mbb{X}(\mbb{G}) = \mbb{G}_{\geq k}$.
\end{definition}

As a shorthand, if $\mbb{G}$ admits an $\mbb{X}$-resolution $\txt{res}_{\mbb{X}}(\mbb{G})=\mbb{G}_{\geq k}$ and $(\mbb{G},\mcal{S})$ is a Ramsey symbol, we will write $\txt{res}_{\mbb{X}}(\mbb{G},\mcal{S}):= (\mbb{G}_{\geq k},\mcal{S}_{\geq k})$ to avoid explicit reference to the value $k$ and the unnecessary introduction of additional variables.  Directly by construction, it is not particularly difficult to see that if $\mbb{G}$ admits an $\mbb{X}$-resolution and $R_{(\mbb{G},\mcal{S})}\left(\mbb{X},\mcal{C}\right)$ is a Ramsey number with a uniform Ramsey symbol, then $R_{\txt{res}_{\mbb{X}}(\mbb{G},\mcal{S})}\left(\mbb{X},\mcal{C}\right)$ is exact. 

\begin{example}
    Let $\mbb{G}=\cb{K_n}_{n \in \N}$ and let $\mbb{X}=\cb{K_6,K_{40}}$.  Then, $\mbb{X}$ is not constituent in $\mbb{G}$, as $K_t/K_6 = \emptyset$ for $0<t < 6$ and $K_t/K_{40} = \emptyset$ for $0<t < 40$.  However, $\mbb{G}$ does admit an $\mbb{X}$-resolution, as the subsets $\mbb{G}_{\geq 6}=\cb{K_n}_{6 \leq n\in \N}\subseteq \mbb{G}$ and $\mbb{V}=\cb{K_6}\subseteq \mbb{X}$ satisfy $\mbb{V} \unlhd \mbb{G}_{\geq 6}$.  Since it is straightforward to see that this is the least value $k$ such that there is a nonempty subset $\mathbb{V}$ of $\mathbb{X}$ satisfying $\mathbb{V}\unlhd\mathbb{G}_{\geq k}$, we have $\txt{res}_{\mbb{X}}(\mbb{G})=\mbb{G}_{\geq 6}$.
\end{example}

$\mbb{X}$-resolutions give us a method to turn non-exact Ramsey numbers into exact ones.

\begin{theorem}\label{exact resolution invariance}
    Let $(\mbb{G},\mcal{S})$ be a Ramsey base and let $(\mbb{X},\mcal{C})$ be a uniform Ramsey symbol.  Then, if $\mbb{G}$ admits an $\mbb{X}$-resolution, $$R_{(\mbb{G},\mcal{S})}\cp{\mbb{X},\mcal{C}} = R_{\txt{res}_{\mbb{X}}(\mbb{G},\mcal{S})}\cp{\mbb{X},\mcal{C}}.$$
\end{theorem}
\begin{proof}
    If $R_{\mbb{G}}\cp{\mbb{X},\mcal{C}}$ does not exist, then we are done, as it is clear that $R_{\txt{res}_{\mbb{X}}(\mbb{G},\mcal{S})}\cp{\mbb{X},\mcal{C}}$ cannot exist either.  So, suppose $\mbb{G}$ admits an $\mbb{X}$-resolution and suppose $R_{(\mbb{G},\mcal{S})}\cp{\mbb{X},\mcal{C}}$ exists.  Then, by definition, $\txt{res}_{\mbb{X}}(\mbb{G})=\mbb{G}_{\geq k}$ for some $0 \leq k< \infty$ and, by the uniformity of the Ramsey symbol, $R_{(\mbb{G},\mcal{S})}\cp{\mbb{X},\mcal{C}}$ is the least $n \in I$ such that, for any $t \geq n$ and any edge-coloring $\rho\in S_t$ of $G_t$, there exists $\pi \in G_t/X_j$ satisfying $\restr{\rho}{\pi^{-1}(X_j)} = \psi_j\circ \pi$ for some $j \in J$.  As $\mbb{G}$ is hereditary by assumption, then $G_m/X_j = \emptyset$ for every $j \in J$ if $m< k$ by definition, as $\mbb{V} \unlhd \mbb{G}_{\geq m}$ must hold for some $\cb{X_j}=\mbb{V} \subseteq \mbb{X}$ if $G_m/X_j$ is nonempty.  Hence, $R_{\mbb{G}}\cp{\mbb{X},\mcal{C}} \geq k$. 
\end{proof}

That is to say, the Ramsey number $R_{(\mbb{G},\mcal{S})}\cp{\mbb{X},\mcal{C}}$ with uniform Ramsey symbol for which $\mbb{G}$ admits an $\mbb{X}$-resolution remains invariant under $\mbb{X}$-resolutions of the Ramsey base.  In fact, a nice partial converse can be shown as well.

\begin{lemma}
    Let $(\mbb{G},\mcal{S})$ be a Ramsey base and let $(\mbb{X},\mcal{C})$ be a uniform Ramsey symbol.  Then, if the Ramsey number $R_{(\mbb{G},\mcal{S})}\cp{\mbb{X},\mcal{C}}$ exists, $\mbb{G}$ admits an $\mbb{X}$-resolution.
\end{lemma}
\begin{proof}
    If the Ramsey number $R_{(\mbb{G},\mcal{S})}\cp{\mbb{X},\mcal{C}}$ exists and has a uniform symbol then, by definition, $R_{(\mbb{G},\mcal{S})}\cp{\mbb{X},\mcal{C}}$ is the least $n \in I$ such that, for any $t \geq n$ and any edge-coloring $\rho\in S_t$ of $G_t$, there exists $\pi \in G_t/X_j$ satisfying $\restr{\rho}{\pi^{-1}(X_j)} = \psi_j\circ \pi$ for some $j \in J$.  As $\mbb{G}$ is hereditary by assumption, then $G_m/X_j = \emptyset$ for every $j \in J$ if $G_m \in \mbb{G}\setminus \txt{res}_{\mbb{X}}(\mbb{G})$ by definition, as $\mbb{V} \unlhd \mbb{G}_{\geq m}$ must hold for some $\cb{X_j}=\mbb{V} \subseteq \mbb{X}$ if $G_m/X_j$ is nonempty.  As the existence of $R_{(\mbb{G},\mcal{S})}\cp{\mbb{X},\mcal{C}}$ requires $G_k/X_j$ to be nonempty for some $j \in J$ and every $k \geq R_{(\mbb{G},\mcal{S})}\cp{\mbb{X},\mcal{C}}$, we are done.
\end{proof}

So, studying exact Ramsey numbers that exist is equivalent to studying Ramsey numbers with a uniform symbol that exist and admit $\mbb{X}$-resolutions. Though this is not necessarily a particularly interesting fact by itself, this is a good illustration of the type of approach we will employ in this paper, which can be summarized as follows: 

\begin{enumerate}
    \item Choose a class of Ramsey numbers that is sufficiently constrained so as to be amicable to study in some sense, or for which some results are known.
    \item Create a device which allows a less constrained class of Ramsey numbers to be converted to or related to the constrained class.
    \item Prove that a property--such as existence or the value--of the less constrained class of Ramsey numbers is preserved under this conversion or relation.
\end{enumerate} 

A slightly less trivial instance of this is presented next.

\section{Indicator Polynomials for Ramsey Numbers of Galois Type}\label{sec2}

In this section, we will consider Ramsey numbers of Galois type and observe a way to understand them via the zero sets of certain polynomials over finite fields.  This will prove instrumental in the later sections of the paper, and will allow us (after applying some conversions and embeddings, which we will develop later on) a basic sort of characterization of the locally-finite type Ramsey numbers.  

In the following definition, we will set ourselves up to exploit the $\F_q$-vector space structure we can impose on $\textrm{Hom}(G,\F_q)$ by defining a certain useful class of polynomial and considering the points in the free vector space $\F_q^{(E(G))}:=\oplus_{e \in E(G)}\F_q$ where they vanish.  Since $\rho \in \textrm{Hom}(G,\F_q)$ produces $v=\oplus_{e \in E(G)}\rho(e) \in \F_q^{(E(G))}$, we will treat vectors $v$ as colorings $\rho$ and vice versa.
\begin{definition}[Subgraph Coloring Indicator Polynomials] \label{inddef} 
    Let $G$, $X$ be finite undirected graphs and let $q$ be a prime power.  If $\psi \in \textrm{Hom}(X,\F_q)$, then the {\em $(X,\psi)$-indicator polynomial of $G$} is the polynomial $p[G,X,\psi] \in \F_q[x_e : e\in E(G)]$ given by 
    $$p[G,X,\psi](x)=\prod_{\pi \in G/X}\left ( 1- \prod_{e \in \pi^{-1}(X)}\left (1-\cp{x_{e}-\psi\cp{\pi(e)}}^{q-1} \right ) \right ).$$
\end{definition}

Recalling that $\F_q$ is a unique factorization domain and that $g^{q-1}=1$ for any nonzero element of $\F_q$, it follows that 
$$\prod_{e \in \pi^{-1}(X)}\left (1-\cp{x_{e}-\psi\cp{\pi(e)}}^{q-1} \right )\neq 0$$ 
\noindent if and only if $x_{e}= \psi\cp{\pi(e)}$ for every $e \in E\cp{\pi^{-1}(X)}$.  That is, indicator polynomials are just a (computationally explicit) way to write an Iverson bracket.  We will use this fact to our advantage in the following theorem.

\begin{lemma}\label{indicator zero}
    Let $G$, $X$ be finite undirected graphs and let $q$ be a prime power.  If $\psi \in \textrm{Hom}(X,\F_q)$, then $p[G,X,\psi](v) = 0$ for $v \in \F_q^{(E(G))}$ corresponding to a coloring $\rho \in \textrm{Hom}(G,\F_q)$ if and only if there exists $\pi \in G/X$ satisfying $$\restr{\rho}{\pi^{-1}(X)}=\psi\circ \pi.$$
\end{lemma}
\begin{proof}
    As $q$ is a prime power, $\F_q$ is a field, and is a unique factorization domain.  So, $p[G,X,\psi](\rho) = 0$ if and only if there is some $\pi \in G/X$ such that 
    
    $$\prod_{e\in \pi^{-1}(X)}\left ( 1-\cp{\rho(e)-\psi\cp{\pi(e)}}^{q-1}\right)=1.$$ 
    
    But, since every nonzero element $a \in \F_q$ satisfies $a^{q-1}=1$, this then implies that $\rho(e)-\psi\cp{\pi(e)} = 0$ must hold for every $e \in \pi^{-1}(X)$, which is to say that $\rho(e) = \psi\cp{\pi(e)}$ must hold for every $e \in \pi^{-1}(X)$.  So, this says $\restr{\rho}{\pi^{-1}(X)}=\psi \circ \pi$.
\end{proof}

If $\mbb{X}=\cb{X_j}_{j \in J}$ is a finite collection (i.e $|J|<\infty$) of finite, undirected graphs, and $\mcal{C}=\cb{\psi_j}_{j\in J}$ is a collection of $\F_q$-valued $q$-edge-colorings $\psi_j\in \textrm{Hom}(X_j,\F_q)$, then we write $p[G,\mbb{X},\mcal{C}]$ to denote the product
$$p[G,\mbb{X},\mcal{C}] :=\prod_{j\in J}p[G,X_j,\psi_j].$$  

This will allow us to use the indicator polynomials of $G$ to discern whether a given $q$-coloring of $G$ contains a colored subgraph isomorphic to a member of a given {\em family} of colored graphs.

\begin{lemma}\label{indicator zero set}
    Let $G$ be a finite undirected graph and let $q$ be a prime power.  If $\mbb{X}=\cb{X_j}_{j \in J}$ is a finite collection of finite undirected graphs and $\mcal{C}=\cb{\psi_j}_{j\in J}$ is a collection of colorings $\psi_j \in \textrm{Hom}(X_j,\F_q)$, then $p[G,\mbb{X},\mcal{C}]$ is identically zero as a polynomial function if and only if for every $\rho\in \textrm{Hom}(G,\F_q)$, there exists $\pi \in G/X_j$  satisfying $\restr{\rho}{\pi^{-1}(X_j)}=\psi_j \circ \pi$ for some $j \in J$.
\end{lemma}
\begin{proof}
    Since $p[G,\mbb{X},\mcal{C}](\rho) = 0$
    if and only if $p[G,X_j,\psi_j](\rho)=0$ for some $j \in J$, Lemma~\ref{indicator zero} delivers the result.
\end{proof}
\begin{example}\label{k4ex}
    Consider $K_4$ with edges numbered as below:
    
    
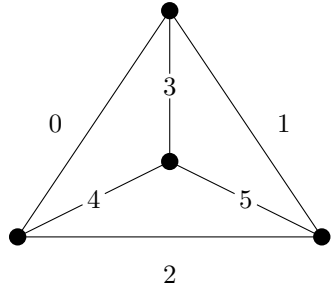
\begin{center}\begin{tikzpicture}
    	[colorstyle/.style={circle, draw=black!100,fill=black!100, thick, inner sep=0pt, minimum size=2 mm},clearstyle/.style={circle, draw=white!100,fill=white!100, thick, inner sep=0pt, minimum size=3 mm}]
    	\node at (0,3)[colorstyle]{}; 
    	\node at (2,0)[colorstyle]{}; 
    	\node at (-2,0)[colorstyle]{}; 
    	\node at (0,1)[colorstyle]{}; 
    	
    	\draw[thin](-2,0)--(2,0)--(0,1)--(-2,0)--(0,3)--(2,0);
    	\draw[thin](0,1)--(0,3);
    	
    	\node at (1.5,1.5){1}; 
    	\node at (-1.5,1.5){0}; 
    	\node at (0,-0.5){2}; 
    
    	\node at (0,2)[clearstyle]{3}; 
    	\node at (1,.5)[clearstyle]{5}; 
    	\node at (-1,.5)[clearstyle]{4}; 
    \end{tikzpicture}\end{center}
    
    Let $\pi:H \rightarrow K_3$ denote any graph isomorphism which sends the subgraph $H$ of $K_4$ induced by the edge set $\ncb{1,3,5}$ to $K_3$.  Then, the polynomial $$\cp{1-x_1x_3x_5}\cp{1-(1-x_1)(1-x_3)(1-x_5)}$$ in $\F_2[x_e:e\in K_4]$ has a zero at any value where $x_1=x_3=x_5=0$ or $x_1=x_3=x_5=1$.  That is, the zeros of this polynomial corresponds to the $2$-edge-colorings of $K_4$ that contains $H$ as a monochromatic subgraph. Considering the other subgraphs of $K_4$ isomorphic to $K_3$ and taking $\mathbb{X}=\ncb{K_3,K_3}$, $\mathcal{C}=\ncb{0_{K_3},1_{K_3}}$, we may proceed to produce (after some slight algebra) the polynomial
    
    \begin{align*}
        p[K_4,\mathbb{X},\mathcal{C}](x) &= \prod_{a\in \F_2}\prod_{\pi \in K_4/K_3}\left ( 1- \prod_{e \in \pi^{-1}(K_3)}\left (1-\cp{x_{e}-a }\right ) \right ) \\ 
        &= \cp{x_0(1-x_3)+x_3(1-x_4)+x_4(1-x_0)}^6\\
        &\quad \cp{x_1(1-x_3)+x_3(1-x_5)+x_5(1-x_1)}^6\\
        &\quad \cp{x_2(1-x_4)+x_4(1-x_5)+x_5(1-x_2)}^6\\
        &\quad \cp{x_0(1-x_1)+x_1(1-x_2)+x_2(1-x_0)}^6
    \end{align*}
    
    \noindent Here, $p[K_4,\mathbb{X},\mathcal{C}](\rho)=0$ if and only if $\rho$ is a 2-coloring of $K_4$ that contains a monochromatic $K_3$.
\end{example}

The above makes it clear that studying the algebraic sets where the polynomials $p[G_i,\mbb{X},\mcal{C}]$ vanish may help us to discern properties related to the existence and value of a Ramsey number $R_{(\mbb{G})}(\mbb{X},\mcal{C})$ of Galois type with maximal base.  We will show that this is precisely the case in the remainder of this section, and will prove that this holds more generally in Section~\ref{sec4}.

\begin{remark}
    It is worth pausing here to stress a distinction between {\em polynomials} and {\em polynomial functions}.  Given a field $\F$,  let $\F[x]$ denote the ring of univariate polynomials in the variable $x$ over $\F$ and let $\mcal{P}(\F,\F)$ denote the polynomial functions $p:\F \rightarrow \F$.  Naturally, both $\F[x]$ and $\mcal{P}(\F,\F)$ are algebras over $\F$, and we can certainly map any $p \in \F[x]$ to $q \in \mcal{P}(\F,\F)$ by the algebra homomorphism $\Phi:\F[x]\rightarrow \mcal{P}(\F,\F)$ defined by the $\F$-action $\Phi(p)(x)=q(x), \enskip \forall x \in \F$.  It is routine to verify that we have $\F[x]/\txt{ker}(\Phi) \cong \mcal{P}(\F,\F)$.  So, for example, the polynomial $x^q-x$ is not the zero polynomial in $\F_q[x]$, but is identically the zero polynomial in $\mcal{P}\cp{\F_q,\F_q}$ since $x^q-x$ factors as $\prod_{a \in \F_q}\cp{x-a}$ in $\F_q[x]$.  More succinctly, $x^q-x \in \txt{ker}(\Phi)$.  Naturally, an analog of the same holds in the multivariate case as well (see, e.g. \cite{clarkchevalley}).  This distinction should be kept in mind to avoid mistakes.  For clarity, we treat polynomials as polynomials, not polynomial functions, unless explicitly stated.
\end{remark}

\begin{remark} 
    Lemma~\ref{indicator zero set} also makes some sense of the choice to use the term '$\mbb{X}$-resolution' in Definition~\ref{exact xres}.  Namely, for a hereditary set $\mbb{G}$ of finite undirected graphs, a finite set $\mbb{X}=\ncb{X_j}_{j\in J}$ of finite undirected graphs, and a set $\mcal{C}=\ncb{\psi_j}_{j\in J}$ of functions $\psi_j\in \textrm{Hom}(X_j,\F_q)$, if $\mbb{G}$ admits an $\mbb{X}$-resolution $\txt{res}_{\mbb{X}}(\mbb{G})=\mbb{G}_{\geq k}$, then it follows from the heredity of $\mbb{G}$ that there are (not necessarily unique) surjective projection homomorphisms $$P_t:\mcal{P}\cp{\F_q^{(E(G_t))},\F_q} \rightarrow \mcal{P}\cp{\F_q^{(E(G_{t-1}))},\F_q}$$  and quotient homomorphisms $$q_t: \mcal{P}\cp{\F_q^{(E(G_t))},\F_q} \rightarrow \mcal{P}\cp{\F_q^{(E(G_t))},\F_q}/ \left \langle p[G_t,\mbb{X},\mcal{C}] \right \rangle$$ which induce, for any $r \geq 1$, the exact sequence of $\F_q$-modules 
    
    $$\mcal{P}\cp{\F_q^{(E(G_{k+r}))},\F_q} \xrightarrow{P_{k+1}\circ\hdots\circ P_{k+r}}\mcal{P}\cp{\F_q^{(E(G_k))},\F_q} \xrightarrow{q_k} 0.$$
    
    \noindent Exactness follows from the surjectivity of the projections and the fact that $p[G_{m},\mbb{X},\mcal{C}]=1$ if and only if $m < k$, yielding $q_k\cp{\mcal{P}\cp{\F_q^{(E(G_k))},\F_q}} = 0$.  So, an $\mbb{X}$-resolution of $\mbb{G}$ induces a corresponding resolution of the $\F_q$-module $\mcal{P}\cp{\F_q^{(E(G_k))},\F_q}$ in the traditional sense.
\end{remark}

\begin{remark}
    The $(\mbb{X},\mcal{C})$-indicator polynomials of a hereditary set of graphs $\mbb{G}$ allow for certain convenient applications of $\mbb{X}$-resolutions.  Namely, by the construction of the indicator polynomials, note that if $t$ is such that $G_t/X_j=\emptyset$ for all $X_j \in \mbb{X}$, then the $(\mbb{X},\mcal{C})$-indicator polynomial of $G_t$ is a constant function equal to $1$.  Using the heredity of $\mbb{G}$, it is somewhat intuitive to see that if $\mbb{G}$ admits an $\mbb{X}$-resolution, then the $(\mbb{X},\mcal{C})$-indicator polynomials of $G_i\in \txt{res}_{\mbb{X}}(\mbb{G})$ may be used to factor the $(\mbb{X},\mcal{C})$-indicator polynomials of $G_t\in \txt{res}_{\mbb{X}}(\mbb{G})$ whenever $t > i$.  Of course, the $(\mbb{X},\mcal{C})$-indicator polynomials of $G_t \in \mbb{G}\setminus \txt{res}_{\mbb{X}}(\mbb{G})$ are also factors of all other indicator polynomials as well, but this is of little interest since these polynomials are constant and identically equal to $1$.

    \begin{theorem}\label{ind poly partial factor}
        Suppose the Ramsey number $R_{\mbb{G}}\left(\mbb{X},\mcal{C}\right)$ is of Galois type and suppose $\mbb{G}$ admits an $\mbb{X}$-resolution.  If $\txt{res}_{\mbb{X}}(\mbb{G})=\mbb{G}_{\geq k}$, then for all $n> k$ and all $\pi \in G_n/G_k$, the polynomial
        
        $$p[\pi(G_k),\mbb{X},\mcal{C}]\cp{x_{\pi^{-1}(G_k)}}$$
        
        \noindent is a factor of $p[G_n,\mbb{X},\mcal{C}](x)$.
    \end{theorem}
    \begin{proof}
        By the heredity of $\mbb{G}$, it follows that $G_n/G_k \neq \emptyset$ for all $n > k$.  So, if $\sigma \in G_n/G_k$, then for any $j \in J$ and any $\eta \in G_k/X_j$, it follows that $\cp{\sigma \circ \eta} \in G_n/X_j$.  So, for any fixed $\pi \in G_n/G_k$, the indicator $p[G_n,X_j,\psi_j](x)$ is divisible by $$\prod_{\eta \in G_k/X_j}\left ( 1- \prod_{e \in \pi^{-1}\cp{\eta^{-1}(X_j)}}\left (1-\cp{x_{e}-\psi_j\cp{(\pi \circ \eta)(e)}}^{q-1} \right ) \right ).$$
        
        This may equivalently be written as $p[\pi(G_k),X_j,\psi_j]\cp{x_{\pi^{-1}(G_k)}}$, and by Definition~\ref{inddef}, we may conclude that this is also a factor of $p[G_n,X_j,\psi_j](x)$.  To produce our desired result, we simply take the product over $j \in J$:
        
        $$\prod_{j \in J}p[\sigma(G_k),X_j,\psi_j]\cp{x_{\sigma^{-1}(G_k)}} = p[\sigma(G_k),\mbb{X},\mcal{C}]\cp{x_{\sigma^{-1}(G_k)}}.$$
    \end{proof}
    
    It is worth noting that this holds for each fixed $\pi \in G_n/G_k$, but it need not be the case that $\prod_{\pi\in G_n/G_k}p[\pi(G_k),\mbb{X},\mcal{C}]\cp{x_{\pi^{-1}(G_k)}}$ divides $p[G_n,\mbb{X},\mcal{C}](x)$.  That is, if $\pi_1,\pi_2 \in G_n/G_k$ and $\eta_1,\eta_2 \in G_k/X_j$, while it is true that $\pi_1\circ \eta_1$ and $\pi_2\circ \eta_2$ are embeddings of $X_j$ in $G_n$, it is possible to have $\pi_1 \circ \eta_1 = \pi_2 \circ \eta_2$ while $\sigma_1 \neq \sigma_2$ and $\eta_1 \neq \eta_2$.  That is, some factors may appear too many times when the product is taken over $G_n/G_k$, analogously to how $(2)(3)=6$ and $(3)(5)=15$ both divide $(2)(3)(5)=30$, but $(6)(15)=90$ does not.
\end{remark}

Returning to the main focus of the section, in the next theorem, we will show that it is not difficult to use the $(\mbb{X},\mcal{C})$-indicators of $\mbb{G}$ to produce a sort of characterization of Ramsey numbers $R_{\mbb{G}}\cp{\mbb{X},\mcal{C}}$ of Galois type with maximal base.  This class of Ramsey numbers, as previously noted, contains all of the classical Ramsey numbers with a prime power number of arguments.

\begin{theorem}\label{maximal base ramsey gal poly formula}
    Suppose the Ramsey number $R_{\mbb{G}}\left(\mbb{X},\mcal{C}\right)$ exists and is of Galois type.  Treating $p[G_i,\mbb{X},\mcal{C}]$ as an element of $\mcal{P}\cp{\F_q^{(E(G_i))},\F_q}$ and writing $i_m=\txt{max}\left \{i\in I: p[G_i,\mbb{X},\mcal{C}]\neq 0 \right\}$ and $i_k=\txt{min}\left \{i\in I: p[G_i,\mbb{X},\mcal{C}]=0 \right\}$, we have $i_{m+1}=R_{\mbb{G}}\left(\mbb{X},\mcal{C}\right) = i_k$.
\end{theorem}
\begin{proof}
    As $R_{\mbb{G}}\left(\mbb{X},\mcal{C}\right)$ is of Galois type, let $q$ be the prime power number of colors. By Definition, since the base of $R_{\mbb{G}}\left(\mbb{X},\mcal{C}\right)$ is maximal, the Ramsey number $R_{\mbb{G}}\left(\mbb{X},\mcal{C}\right)=n$ if $n\in I$ is the lowest order element of $I$ such that, for $t \geq n$ and any $\F_q$-valued $m$-edge-coloring $\rho \in \textrm{Hom}\cp{G_t,\F_q}$ of $G_t$, there exists $\pi \in G_t/X_j$ satisfying $\restr{\rho}{\pi^{-1}(X_j)} = \psi_j \circ \pi$ for some $X_j \in \mbb{X}$.  But, as $p[G_t,\mbb{X},\mcal{C}]$ is identically the zero polynomial in $\mcal{P}\cp{\F_q^{(E(G_t))},\F_q}$ if and only if for every $\rho \in \textrm{Hom}(G_t,\F_q)$ there exists $\pi \in G_t/X_j$ satisfying $\restr{\rho}{\pi^{-1}(X_j)}=\psi_j\circ \pi$ for some $j \in J$, it follows that that $R_{\mbb{G}}\left(\mbb{X},\mcal{C}\right) \geq i_{m+1}$ if $p[G_{i_m},\mbb{X},\mcal{C}]$ is not identically the zero polynomial in $\mcal{P}\cp{\F_q^{(E(G_{i_m}))},\F_q}$. In the reverse, it follows similarly that $p[G_{t},\mbb{X},\mcal{C}]$ is identically the zero polynomial in $\mcal{P}\cp{\F_q^{(E(G_{t}))},\F_q}$ for every $t \geq n$ and $R_{\mbb{G}}\left(\mbb{X},\mcal{C}\right) \leq i_k$.  So, $i_{m+1} \leq R_{\mbb{G}}\left(\mbb{X},\mcal{C}\right) \leq i_k$.  The result follows from the assumption that $I$ is well-ordered by noting that $i_{m+1}=i_k$.
\end{proof}

As a corollary, we get the following:

\begin{corollary}\label{ramsey ind poly formula}
    Suppose the Ramsey number $R_{\mbb{G}}\left(\mbb{X},\mcal{C}\right)$ is of Galois type with maximal base and suppose $R_{\mbb{G}}\left(\mbb{X},\mcal{C}\right)=n$.  Then, $p[G_n,\mbb{X},\mcal{C}] \in \mcal{I}$, where $\mcal{I}$ is the ideal generated by the polynomials $\cb{x_e^q-x_e: e \in E(G_n)}$. 
\end{corollary}
\begin{proof}
    If $R_{\mbb{G}}\left(\mbb{X},\mcal{C}\right)$ exists and is of Galois type with maximal base and $R_{\mbb{G}}\left(\mbb{X},\mcal{C}\right)=n$, then $p[G_n,\mbb{X},\mcal{C}](z)=0$ for all $z \in \F_q^{(E(G_n))}$ by Theorem~\ref{maximal base ramsey gal poly formula}.  So, since $p[G_n,\mbb{X},\mcal{C}]$ is not the zero polynomial but is identically zero as a polynomial function, it must lie in the kernel of the evaluation map $\Phi$, i.e. the algebra homomorphism $\Phi:\F_q\left [x_e: e\in E(G_n)]\right ] \rightarrow \mcal{P}\cp{\F_q^{(E(G_n))},\F_q}$ that sends a polynomial to the corresponding polynomial function of minimal degree with which it agrees point-wise (see, e.g. \cite{clarkchevalley}).  But, since $\txt{ker}(\Phi)= \mcal{I}$, where $\mcal{I}$ is the ideal generated by the polynomials $\cb{x_e^q-x_e: e \in E(G_n)}$, this then says that $p[G_n,\mbb{X},\mcal{C}] \in \mcal{I}$.
\end{proof}

Put another way, if the Ramsey number $R_{\mbb{G}}\left(\mbb{X},\mcal{C}\right)$ exists and is of Galois type and $\langle p[G_i,\mbb{X},\mcal{C}] \rangle $ is the ideal in $\F_q[x_e : e \in E(G_i)]$ generated by the $(\mbb{X},\mcal{C})$-indicator polynomial of $G_i$, then the algebraic set 

$$V\cp{\langle p[G_i,\mbb{X},\mcal{C}] \rangle} = \cb{x \in \F_q^{(E(G_i))} : p(x)=0, \es \forall p \in \langle p[G_i,\mbb{X},\mcal{C}] \rangle}$$ 

\noindent must satisfy $V\cp{\langle p[G_i,\mbb{X},\mcal{C}] \rangle} = \F_q^{(E(G_i))}$ for all values of $i \geq R_{\mbb{G}}\left(\mbb{X},\mcal{C}\right)$.  This makes it clear that everything we have done so far can be extended to Ramsey numbers of Galois type without maximal base without much effort:  we need only check that $V\cp{\langle p[G_i,\mbb{X},\mcal{C}] \rangle}$ contain all admissible colorings $S_i$.  We can state this precisely by adopt the usual convention of writing $V\cp{\mcal{I}}$ to denote the algebraic set on which the members of an ideal $\mcal{I}$ vanish and, for an algebraic set $U$, we will write $I(U)$ to denote the ideal of all polynomials that vanish on $U$.

\begin{corollary}\label{ramsey gal poly formula}
    Suppose the Ramsey number $R_{(\mbb{G},\mcal{S})}\left(\mbb{X},\mcal{C}\right)$ exists and is of Galois type.  Then, $R_{(\mbb{G},\mcal{S})}\left(\mbb{X},\mcal{C}\right) = n$ if and only if $n$ is the least value such that 
    $$S_n \subseteq V\left(\langle p[G_i,\mbb{X},\mcal{C}] \rangle\right)$$ or, equivalently, such that $$I(S_n) \supseteq \langle p[G_i,\mbb{X},\mcal{C}] \rangle.$$
\end{corollary}

\begin{example}
    Consider the indicator polynomial $p\left [K_4,\left \{K_3,K_3\right \},\left \{ 0_{K_3}, 1_{K_3}\right \}\right ]$ given in Example~\ref{k4ex}.  It is routine to check that the coloring $\rho =\ba{cccccc}{1 & 1 & 0 & 0 & 1 & 0}$ yields $p\left [K_4,\left \{K_3,K_3\right \},\left \{ 0_{K_3}, 1_{K_3} \right \}\right](\rho) = 1$.  So, $$V\cp{\langle p\left [K_4,\left \{K_3,K_3\right \},\left \{ 0_{K_3}, 1_{K_3} \right \}\right]\rangle} \not\supset  F_2^{(E(K_4))}$$ and we may conclude $R(3,3)>4$.
\end{example}

Via Theorem~\ref{maximal base ramsey gal poly formula} and Corollary~\ref{ramsey gal poly formula}, we know that studying Ramsey numbers of Galois type without maximal base is not, on a conceptual level, any harder than studying Ramsey numbers of Galois type with maximal base. This is mainly due to the fact that the $(\mbb{X},\mcal{C})$-indicator polynomials do not depend on the set of admissible colorings $\mcal{S}$.  We can basically use the same trick to show that any Ramsey numbers with bases of finite type and uniform symbols of finite type can be converted to an equivalent Ramsey number of Galois type, or--more generally--that any Ramsey number with a base and symbol of locally finite type is `locally' like a Ramsey number of Galois type. In the next section, we do this.

\section{Indicator Polynomials for Ramsey Numbers of Locally Finite Type}\label{sec3}

We begin by adapting the locally finite case to a local analog of the Galois case.  Among other things, this allows any Ramsey numbers with bases of finite type and uniform symbols of finite type to be converted to an equivalent Ramsey number of Galois type.  In either case, the conversion is simple and just requires choosing a few injections.

\begin{theorem}\label{gal embed invariance}
    Let $R_{(\mbb{G},\mcal{S})}(\mbb{X},\mcal{C})$ be a Ramsey number with a base and symbol of locally finite type and let $f_i:A_i\rightarrow \F_{q_i}$ be any sequence of injections, where $q_i$ is any appropriately large prime power.  Then, setting $\mcal{S}^\prime=(S_i^\prime)$ and $\mcal{C}^\prime=(C_i^\prime)$, where $S_i\prime = \ncb{f_i \circ \rho: \rho \in S_i}$ and $C_i^\prime = \ncb{f_i \circ \psi: \psi \in C_i}$, we have $R_{(\mbb{G},\mcal{S})}(\mbb{X},\mcal{C})=R_{(\mbb{G},\mcal{S}^\prime)}(\mbb{X},\mcal{C}^\prime)$.
\end{theorem}
\begin{proof}
    By Definition~\ref{GRN}, we know that $R_{(\mbb{G},\mcal{S}^\prime)}(\mbb{X},\mcal{C}^\prime)= n$ if and only if $n$ is the minimum element of $I$ such that for every $t \geq n$ and any $\phi S_t^\prime \subseteq \textrm{Hom}\cp{G_t,\F_{q_t}}$, there is some $j \in J_t$ such that there exists $\pi \in G_t/X_j$ satisfying, for the corresponding coloring $\eta_j \in C_t^\prime$, $\restr{\phi}{\pi^{-1}(X_j)}=\eta_j \circ \pi$.  But, by definition, if $\phi \in S_t^\prime$ and $\eta_j \in C_t^\prime$, then $\phi = f_t \circ \rho$ for some $\rho \in S_t$ and $\eta_j = f_t \circ \psi_j$ for $\psi_j \in C_t$.  So, we then have $\restr{f_t \circ \rho}{\pi^{-1}(X_j)}=f_t \circ \psi_j \circ \pi$.  As $f_t$ is an injection, it has a left inverse which we may apply to produce $\restr{\rho}{\pi^{-1}(X_j)}=\psi_j \circ \pi$.  This is precisely the condition required for $R_{(\mbb{G},\mcal{S})}(\mbb{X},\mcal{C}) = n$ as well.
\end{proof}

Notice that, for Ramsey numbers with bases and symbols of finite type, the sequence of injections can be chosen to be constant, rendering $R_{(\mbb{G},\mcal{S}^\prime)}(\mbb{X},\mcal{C}^\prime)$ a Ramsey number of Galois type and facilitating the conversion mentioned before the theorem.  More generally, Theorem~\ref{gal embed invariance} tells us that it is possible to apply everything we have done so far to the {\em much larger} class of Ramsey numbers with a base and symbol of locally finite type.  The requirement that a Ramsey number be of Galois type imposed in the last section largely served only as an illustration of how this process should work in a simpler case.

In light of Theorem~\ref{gal embed invariance}, we can show that Corollary~\ref{ramsey ind poly formula} holds more generally.

\begin{theorem}\label{ramsey ind poly formula locally finite}
    Suppose the Ramsey number $R_{\mbb{G}}\left(\mbb{X},\mcal{C}\right)$ is of locally finite type with maximal base and suppose $R_{\mbb{G}}\left(\mbb{X},\mcal{C}\right)=n$.  Then, for all $t \geq n$, $p[G_t,\mbb{X},\mcal{C}^\prime] \in \mcal{I}_t$, where $\mathcal{C}^\prime$ is as in Theorem~\ref{gal embed invariance} and $\mcal{I}_t$ is the ideal in $\F_{q_t}[x_e: e \in E(G_t)]$ generated by the polynomials $\cb{x_e^{q_t}-x_e: e \in E(G_t)}$. 
\end{theorem}
\begin{proof}
    Note that $R_{\mbb{G}}\left(\mbb{X},\mcal{C}\right)=n$ if and only if $p[G_t,\mbb{X},\mcal{C}^\prime](z)=0$ for all $z \in \F_{q_t}^{(E(G_t))}$, $t \geq n$, as shown in the proof of Theorem~\ref{gal embed invariance}.  So, since $p[G_t,\mbb{X},\mcal{C}^\prime]$ is not the zero polynomial but is identically zero as a polynomial function, it must lie in the kernel of the evaluation map $\Phi_t$, i.e. the algebra homomorphism $\Phi_t:\F_{q_t}\left [x_e: e\in E(G_t)]\right ] \rightarrow \mcal{P}\cp{\F_{q_t}^{(E(G_t))},\F_{q_t}}$ that sends a polynomial to the corresponding polynomial function of minimal degree with which it agrees point-wise.  But, since $\txt{ker}(\Phi_t)= \mcal{I}_t$, where $\mcal{I}_t$ is the ideal generated by the polynomials $\cb{x_e^{q_t}-x_e: e \in E(G_t)}$, this then says that $p[G_t,\mbb{X},\mcal{C}^\prime] \in \mcal{I}_t$.
\end{proof}

Analogously to Corollary~\ref{ramsey gal poly formula}, we then reach the following result immediately:

\begin{theorem}\label{ramsey characterization locally finite}
    Let $R_{(\mbb{G},\mcal{S})}(\mbb{X},\mcal{C})$ be a generalized Ramsey number with a base of locally finite type and symbol of locally finite type and let $\mathcal{S}^\prime$, $\mathcal{C}^\prime$ be as in Theorem~\ref{gal embed invariance}.  Then, $R_{(\mbb{G},\mcal{S})}(\mbb{X},\mcal{C}) =n$ if and only if $n$ is the least value such that 
    $$S_t^\prime \subseteq V\cp{ \big \langle p[G_t,\mbb{X},C_t^\prime] \big \rangle }$$
    
    \noindent or, equivalently, such that
    
    $$I\cp{S_t^\prime} \supseteq \big \langle p[G_t,\mbb{X},C_t^\prime] \big \rangle.$$
\end{theorem}

The above result says that--locally at each level $i \in I$, at least--Ramsey numbers of locally finite type are, really, no harder to deal with conceptually than those of Galois type.
\section{Some Structural Results for Ramsey Numbers}\label{sec4}

In this section, we will provide some structural results regarding the manipulation of Ramsey numbers.  This will highlight a connection that ties in to the previous results nicely.  The next two lemmas regard the extension and restriction of components of either the symbol or base of a Ramsey numbers with a base of finite type and uniform symbol of finite type.

\begin{lemma}\label{gal type symbol extend}
    Let $R_{(\mbb{G},\mcal{S})}(\mbb{X},\mcal{C})$ be a Ramsey number with a base of finite type and uniform symbol of finite type.  Suppose that $(\mbb{Y},\mcal{W})$ is a uniform Ramsey symbol of finite type and that $(\mbb{X},\mcal{C}) \subseteq (\mbb{Y},\mcal{W})$.  Then if $R_{(\mbb{G},\mcal{S})}(\mbb{X},\mcal{C})$ exists, $R_{(\mbb{G},\mcal{S})}(\mbb{Y},\mcal{W})$ exists and $$R_{(\mbb{G},\mcal{S})}(\mbb{Y},\mcal{W})\leq R_{(\mbb{G},\mcal{S})}(\mbb{X},\mcal{C}).$$
\end{lemma}
\begin{proof}
    As both $R_{(\mbb{G},\mcal{S})}(\mbb{X},\mcal{C})$ and $R_{(\mbb{G},\mcal{S})}(\mbb{Y},\mcal{W})$ are Ramsey numbers with a base of finite type and uniform symbol of finite type, we may treat them as Ramsey numbers of Galois type with colorings taking values over the same finite field via Theorem~\ref{gal embed invariance}.  So, as $(\mbb{X},\mcal{C}) \subseteq (\mbb{Y},\mcal{W})$ and $|(\mbb{X},\mcal{C})| \leq |(\mbb{Y},\mcal{W})|$, the $(\mbb{X},\mcal{C})$-indicator polynomial of $G_i$ is a factor of the $(\mbb{Y},\mcal{W})$-indicator polynomial of $G_i$.  More precisely, if $J_1$ is the index set of $(\mbb{X},\mcal{C})$ and $J_2$ is the index set of $(\mbb{Y},\mcal{W})$ we may assume, without loss of generality, that $J_1 \subseteq J_2$.  So, writing $\mbb{Y}=\cb{Y_j}_{j \in J_2}$ and $\mcal{W}=\cb{\phi_j}_{j \in J_2}$, we have
    
    \begin{align*}
        p[G_i,\mbb{Y},\mcal{W}](x) &= \prod_{j\in J_2}p[G_i,Y_j,\phi_j](x)\\
        &=\left (\prod_{j \in J_1} p[G_i,Y_j,\phi_j](x)\right)\left (\prod_{j \in J_2\setminus J_1} p[G_i,Y_j,\phi_j](x)\right) \\
        &= p[G_i,\mbb{X},\mcal{C}](x)\left (\prod_{j \in J_2\setminus J_1} p[G_i,Y_j,W_j](x)\right).
    \end{align*}
    
    If $R_{(\mbb{G},\mcal{S})}(\mbb{X},\mcal{C})$ exists, then by Corollary~\ref{ramsey gal poly formula}, there is a value $n \in I$ such that for any $t \in I$ satisfying $t \geq n$, we have  $p[G_t,\mbb{X},\mcal{C}]\cp{\rho}=0$ for any $\rho \in S_t$.  By the above partial factorization, this then says that $p[G_t,\mbb{Y},\mcal{W}]\cp{\rho}=0$ for any $\rho \in S_t$ as well.  So, $R_{(\mbb{G},\mcal{S})}(\mbb{Y},\mcal{W})$ exists and $R_{(\mbb{G},\mcal{S})}(\mbb{Y},\mcal{W})\leq R_{(\mbb{G},\mcal{S})}(\mbb{X},\mcal{C})$.
\end{proof}

In particular, Lemma~\ref{gal type symbol extend} tells us that if $Q=(\mbb{X},\mcal{C})$ is the symbol of a Ramsey number with a base of finite type and uniform symbol of finite type and $\emptyset \neq Y\subseteq Q$ is a uniform Ramsey symbol of finite type, then 

$$R_{(\mbb{G},\mcal{S})}(Q) \leq R_{(\mbb{G},\mcal{S})}(Q \cap Y).$$
That is, Ramsey numbers with a base of finite type and uniform symbol of finite type {\em grow larger} under {\em intersections}.  In the next result, we show that the opposite happens when changes are made to a portion of the base, i.e. Ramsey numbers with a base of finite type and uniform symbol of finite type grow smaller under intersections of admissible coloring sets.

\begin{lemma}\label{gal type base restrict}
    Let $R_{(\mbb{G},\mcal{S})}(\mbb{X},\mcal{C})$ be a Ramsey number with a base of finite type and uniform symbol of finite type.  Suppose $L_i \subseteq S_i$ for all $i \in I$.  Then, if $R_{(\mbb{G},\mcal{S})}(\mbb{X},\mcal{C})$ exists, $R_{(\mbb{G},\mcal{L})}(\mbb{X},\mcal{C})$ exists and $$R_{(\mbb{G},\mcal{L})}(\mbb{X},\mcal{C})\leq R_{(\mbb{G},\mcal{S})}(\mbb{X},\mcal{C}).$$
\end{lemma}
\begin{proof}
    As both $R_{(\mbb{G},\mcal{S})}(\mbb{X},\mcal{C})$ and $R_{(\mbb{G},\mcal{L})}(\mbb{X},\mcal{C})$ are Ramsey numbers with a base of finite type and uniform symbol of finite type, we may treat them as Ramsey numbers of Galois type with colorings taking values over the same finite field by Theorem~\ref{gal embed invariance}.  So, note that if $R_{(\mbb{G},\mcal{S})}(\mbb{X},\mcal{C})$ exists, then by Corollary~\ref{ramsey gal poly formula}, there is a value $n \in I$ such that for any $t \in I$ satisfying $t \geq n$, we have  $p[G_t,\mbb{X},\mcal{C}](\rho)=0$ for any $\rho \in S_t$.  Since both $R_{(\mbb{G},\mcal{S})}(\mbb{X},\mcal{C})$ and $R_{(\mbb{G},\mcal{S})}(\mbb{Y},\mcal{W})$ are Ramsey numbers of Galois type with the same Ramsey symbol, and as $L_i \subseteq S_i$ for all $i \in I$, it must then hold that $p[G_t,\mbb{X},\mcal{C}](\rho)=0$ for any $\rho \in L_t$ as well.  Hence, $R_{(\mbb{G},\mcal{L})}(\mbb{X},\mcal{C})$ exists and $R_{(\mbb{G},\mcal{L})}(\mbb{X},\mcal{C})\leq R_{(\mbb{G},\mcal{S})}(\mbb{X},\mcal{C})$.
\end{proof}

Together, the above lemmas tell us that Ramsey numbers with a base of finite type and uniform symbol of finite type behave somewhat intuitively, in the sense that {\em extending} the Ramsey symbol or {\em restricting} the admissible coloring functions in the Ramsey base {\em decreases} the value of a given Ramsey number, provided it exists to begin with. That is, $R_{(\mathbb{G},\cdot)}(\mathbb{X},\mathcal{C})$ is monotonic with respect to inclusion and $R_{(\mathbb{G},\mathcal{S})}(\cdot)$ is monotonic with respect to reverse inclusion.

This suggests that Ramsey numbers share more than a superficial similarity to the classic algebraic geometry ideal-variety interplay described in the previous section.  That is, Ramsey numbers are somewhat like \textit{Galois connections}; i.e., pairs of functions  $F:A\rightarrow B$ and $G:B\rightarrow A$ between posets $A$ and $B$ such that $F(a) \leq b$ if and only if $a \leq G(b)$. Consider the following informal side-by-side comparison, where the symbol "$\subseteq$" in the second line is taken to mean "factors through a restriction of":

\noindent \begin{tabular}{ccccccc}
    $a$ &$\leq$ &$G(b)$ &$\Leftrightarrow$  &$F(a)$ &$\leq$ &$b$ \\
    $\begin{bmatrix}\textrm{some coloring }\\ \psi_j \in C_n\end{bmatrix}$ &$\subseteq$ &$\begin{bmatrix}\textrm{any coloring}\\\rho \in S_n\end{bmatrix}$  &$\Leftrightarrow$ &$R_{(\mathbb{G},\mathcal{S})}(\mathbb{X},\mathcal{C})$&$\leq$ &$n$ \\
\end{tabular}

We conjecture that developing this correspondence more fully could give insights into the 'correct' formulation of a general theory of Ramseyian phenomena.  In the next section, we will try to offer some evidence that doing so might be interesting and worthwhile.

\section{Ramseyian Formulations of Classically Non-Ramseyian Problems}\label{sec6}
Now, some applications. We consider some familiar problems from other mathematical disciplines.  In particular, problems which involve the assignment of distinct values belonging to some common set to pairs of a finite set of objects generally admit some form of a coloring interpretation.  Most notably, {\em distances} between pairs of objects may be regarded in this manner by attaching data to the vertex set of a graph.

\begin{definition}[Metrical Colorings]
    Let $G$ be a finite, undirected graph, let $\cp{Y,d}$ be a metric space, let $\Theta:V(G) \rightarrow Y$ be a fixed injection, and write $A[\Theta] = \cb{d\cp{\Theta(v),\Theta(w)} : v, w \in V(G)}$.  Then, for any set $A$ such that $A[\Theta] \subseteq A$, the {\em $A$-valued, $Y$-metrical coloring of $G$ induced by $\Theta$} is the $A$-valued edge-coloring of $G$, denoted $d_\Theta$, which sends the edge incident to vertices $v$ and $w$ to $d\cp{\Theta(v),\Theta(w)} \in A$.
\end{definition}

A useful, somewhat unexpected consequence of allowing the admissible coloring set $\mcal{S}$ to be smaller than the set $\textrm{Hom}(G,A)$ of all $A$-valued colorings for a given graph $G$ is that this allows the value of a Ramsey number to be used like an indicator function, eventually identifying a desired object if its value is finite and never obtaining the object if not.  Let's illustrate this by re-framing a well-known number-theoretic result in Ramseyian form.  For notation, we will say a set $\ncb{a_1,\hdots,a_n,a_{n+1}} \subseteq \Z$ that satisfies $a_{i+1}-a_i=k$ for all $i \in \ncb{1,\hdots,n}$ is an {\em arithmetic progression of length $n$ with gap $k$}.

\begin{theorem}[Green-Tao \cite{green2008primes}]
    For any $t \in \N$, there exists an arithmetic progression of length $t$ in the primes.
\end{theorem}

The Green-Tao Theorem is equivalent to the existence of a family of Ramsey numbers. 
 First, an intermediate Ramsey-theoretic statement:

\begin{theorem}\label{ramsey gap k length t+1 AP}
    Let $\mbb{G}=\cb{K_{i+1}}_{i \in \N}$ and let $\Z_{\geq 0}$ have the discrete topology with $d(x,y)=|x-y|$ so that $\cp{\Z_{\geq 0},d}$ is a metric space.  Let $p_n$ denote the $n^{th}$ prime, fix $m \in \N$, and let $\Theta[i,m]:V(K_{i+1}) \rightarrow \N$ be any fixed injection such that $\Theta[i,m]^{-1}\cp{ \left \{p_m,\hdots,p_{m+i},p_{m+i+1} \right \} } = V(K_{i+1})$ for each $i \in \N$.  Set $A_{(i,m)} = \ncb{n \in \N:n<p_{m+i+1}}$, let $d_{\Theta[i,m]}$ be the $A_{(i,m)}$-valued $(\Z,d)$-metrical coloring of $K_{i+1}$ induced by $\Theta[i,m]$, set $S_i(m)=\cb{d_{\Theta[i,m]}}$, and let $\mcal{S}(m)=\cb{S_{i}(m)}_{i \in \N}$.  Then, letting $t,k \in \N$ and denoting by $P_t$ the path graph with $t$ edges (i.e. so $P_1=K_2$), the existence of an arithmetic progression of length $t$ with gap $k$ in the primes greater than or equal to $p_m$ is equivalent to the existence of the Ramsey number with symbol of locally finite type and uniform base of finite type $R_{{(\mbb{G},\mcal{S}(m))}}\cp{P_t,k_{P_t}}$.
\end{theorem}
\begin{proof}
    The existence of the Ramsey number $R_{{(\mbb{G},\mcal{S}(m))}}\cp{P_t,k_{P_t}}$ is equivalent, by definition, to the existence of $n \in \N$ such that $n$ is the least element of $\N$ such that for all $r \geq n$ there exists $\pi \in G_{r+1}/P_t$ such that $\restr{d_{\Theta[r,m]}}{\pi^{-1}(P_t)}=k_{P_t} \circ \pi$.
    
    But, this then says that $d_{\Theta[r,m]}(e) = k$ for all $e \in \pi^{-1}(P_t)$.  That is, take the subgraph $H$ of $K_{r+1}$ induced by the edge set $\pi^{-1}(P_t) \subseteq E(K_{r+1})$ and choose the vertex $v_1 \in V(H)$ such that $v_1$ is of degree $1$ and $\Theta[r](v_1)<\Theta[r](w)$ for any other $w \in V(H)$ of degree $1$.  Let $e_1$ denote the edge incident to $v_1$ and inductively label the edges and vertices of $H$ so that the edge $e_i$ is incident to the vertices $v_{i}$ and $v_{i+1}$.  Then, we have ${d_{\Theta[r,m]}(e_i) = d\cp{\Theta(v_i),\Theta(v_{i+1})} = k}$.  Since $\pi^{-1}(P_t)$ is a path in $K_{r+1}$, this then yields a sequence of $t+1$ primes corresponding to the $t+1$ vertices in $K_{r+1}$ incident to the edges belong to $\pi^{-1}(P_t)$ which are successively of distance $k$ and all greater than or equal to $p_m$.  As there are exactly two integers of distance $k$ from any given integer, and as $\Theta[r,m](v_1) < \Theta[r,m](v_{t+1})$, it then follows that $\Theta[r,m](v_1),\Theta[r,m](v_1)+k,\hdots,\Theta[r,m](v_1)+tk$ is an arithmetic progression of length $t$ with gap $k$ in the primes greater than or equal to $p_m$.
\end{proof}

Theorem~\ref{ramsey gap k length t+1 AP} obviously admits an interpretation in terms of the Green-Tao Theorem.

\begin{corollary}\label{GT ramsey}
    Let $\mbb{G}$ and $\mcal{S}(m)$ be as in Theorem~\ref{ramsey gap k length t+1 AP} and, for a fixed $t \in \N$, set $\mbb{X}(t)=\cb{\mbb{X}_i(t)}_{i \in \N}$ and $\mcal{C}(t)=\cb{C_i(t)}_{i\in \N}$, where $\mbb{X}_i(t)=\cb{(P_t)_j}_{0 <j <i}$ and $\mcal{C}_i(t)=\cb{j_{(P_t)_j}}_{0<j<i}$.  The Ramsey numbers $R_{(\mbb{G},\mcal{S}(1))}\cp{\mbb{X}(t),\mcal{C}(t)}$ exist for all values of $t \in \N$.
\end{corollary}

Let's take a look at another example of this phenomenon in the context of a conjecture that is, as of yet, unproven.

\begin{conjecture}[Twin Prime Conjecture]
    There are infinitely many values $n\in \N$ such that $p_{n+1} - p_n = 2$.
\end{conjecture}

It is easy to see that the Twin Prime conjecture is equivalent to the existence of a family of Ramsey numbers which may be constructed in virtually the same manner as the Green-Tao theorem.

\begin{theorem}\label{twin prime}
    Let $\mbb{G}$ and $\mcal{S}(m)$ be as in Theorem~\ref{ramsey gap k length t+1 AP}.  The existence of the of the Ramsey numbers $R_{\cp{\mbb{G},\mcal{S}(m)}}\cp{\ncb{K_2},\ncb{2_{K_2}}}$ for all values of $m \in \N$ is equivalent to the Twin Prime conjecture.
\end{theorem}

Likewise, recall Zhang's landmark 2013 result \cite{zhang2014bounded} that established the existence of infinitely many prime pairs with common gap.

\begin{theorem}[Zhang]\label{zhang}
    For some some $N \in  \N$, there are infinitely many values $n \in \N$ such that $p_{n+1}-p_n = 2N$.
\end{theorem}

\noindent As is well-known, Zhang's remarkable proof verifies a particular case of Polignac's conjecture \cite{tattersall2005elementary}.

\begin{conjecture}[Polignac]
    For every $m\in N$ there are infinitely many values $n\in N$ such that $p_{n+1}-p_n = 2m$.
\end{conjecture}

Naturally, the case $m=1$ in Polignac's conjecture yields the Twin Prime conjecture, and Zhang's result shows that a weakening of Polignac's conjecture holds for some value $m$.  It is somewhat more involved than the constructions used above (but still, ultimately not too difficult) to find some families of Ramsey numbers whose existence is equivalent to Polignac's conjecture.  The following formulation will suffice for our purposes. 

\begin{theorem}\label{polignac}
    Let $\mbb{G}$ and $\Theta[i,m]$ be as in Theorem~\ref{ramsey gap k length t+1 AP} and set $\mbb{X}=\ncb{K_2}$ and $\mcal{W}(t)=\ncb{(2t)_{K_2}}$.  For each $i \in \N$, fix $\sigma_i \in K_{i+1}/P_i$ such that the subgraph of $K_{i+1}$ induced by the edge set $\sigma_i^{-1}(P_i)$ has  $\Theta[i,m]^{-1}(p_{m+r})$ adjacent to $\Theta[i,m]^{-1}(p_{m+r+1})$ for every $r \in \ncb{0,\hdots,i}$.  Let $L_i(m) \subseteq \textrm{Hom}(K_{i+1},A_{(i,m)})$ denote the set of edge-colorings $\eta$ that satisfy $\restr{\eta}{\sigma_i^{-1}(P_i)} = d_{\Theta[i,m]}$.  Then, setting $\mcal{L}(m) = \ncb{L_i(m)}_{i\in \N}$, the existence of the Ramsey numbers $R_{\cp{\mbb{G},\mcal{L}(m)}}\cp{\mbb{X},\mcal{W}(t)}$ for all values of $t,m \in \N$ is equivalent to Polignac's conjecture.
\end{theorem}
\begin{proof}
    To begin, note that for any choice of $m,t \in \N$, the definition of $L_i(m)$ yields that there exists a coloring $\eta \in L_i(m)$ that satisfies 
    
    $$\restr{\eta}{E(K_{i+1})\setminus \sigma_i^{-1}(P_i)} = \restr{r_{K_{i+1}}}{E(K_{i+1})\setminus \sigma_i^{-1}(P_i)}$$ 
    
    \noindent for any $r \in A_{(i,m)}$.  If $A_{(i,m)}=\ncb{2t}$, then we are done.  If not, then by choosing $r \neq 2t$, it follows that if there exists $\pi \in K_{i+1}/K_2$ such that $\restr{\eta}{\pi^{-1}(K_2)} = (2t)_{K_2} \circ \pi$, then $\pi^{-1}(K_2) \subseteq \sigma_i^{-1}(P_i)$.  So, since $\restr{\eta}{\sigma_i^{-1}(P_i)}=d_{\Theta[i,m]}$, the desired conclusion follows immediately.
\end{proof}

It's obvious that other related conjectures can also be posed in a Ramseyain form with only a small degree of effort.  Per Theorem~\ref{polignac}, it's easy to see that Zhang's result is no exception, and is equivalent to the following Ramsey-theoretic statement:

\begin{theorem}\label{zhang ramsey}
    For some $N\in \N$, the Ramsey numbers $R_{\cp{\mbb{G},\mcal{L}(m)}}\cp{\mbb{X},\mcal{W}(N)}$ as in Theorem~\ref{polignac} exist for every $m \in \N$.
\end{theorem}

On one hand, this demonstrates that determining whether or not a Ramsey number $R_{(\mathbb{G},\mathcal{S})}(\mathbb{X},\mathcal{C})$ exists is, at least, as hard as the Twin Prime conjecture or Polignac's conjecture.  On the other hand, it is not entirely outlandish to imagine that there might be at least a slight possibility that some useful insight into problems of this nature could be gained by this sort of a conversion--perhaps by generating upper bounds on some transformed Ramseyian formulation of the problem.  Regardless, it certainly shows that the version of Ramsey numbers addressed in this paper capture more than the classical or generalized forms of Ramsey numbers that are often studied.  Complicated, highly nontrivial nonclassical Ramsey numbers do indeed exist.

\section{Concluding Remarks}\label{sec7}

In this paper, we have developed a significantly more general notion of classical Ramsey numbers (extending most other generalizations) and performed some basic characterizations of them using a few simple algebraic tools.  This generalization allowed a better vantage to take stock of various Ramseyian problems and notice some similarities to Galois connections that hint that a better, less-cumbersome formulation might be possible.  To conclude, it was demonstrated that the Green-Tao Theorem, the Twin Prime conjecture, Zhang's bounded prime gap theorem, and Polignac's conjecture can be viewed as statements about Ramsey numbers.  This may offer potential avenues to explore Ramsey-theoretic interpretations of classical non-Ramseyian problems, and at least shows that the generalization of Ramsey numbers introduced captures a larger range of phenomena than one might initially expect.

\bibliographystyle{spmpsci} 
\bibliography{ref}  

\begin{thebibliography}{10}
\providecommand{\url}[1]{{#1}}
\providecommand{\urlprefix}{URL }
\expandafter\ifx\csname urlstyle\endcsname\relax
  \providecommand{\doi}[1]{DOI~\discretionary{}{}{}#1}\else
  \providecommand{\doi}{DOI~\discretionary{}{}{}\begingroup
  \urlstyle{rm}\Url}\fi

\bibitem{bollobás1998modern}
Bollob{\'a}s, B., s, B., Axler, S., Bollobas, B., Gehring, F., Halmos, P.:
  Modern Graph Theory.
\newblock Graduate Texts in Mathematics. Springer New York (1998).
\newblock \urlprefix\url{https://books.google.com/books?id=SbZKSZ-1qrwC}

\bibitem{burr1974generalized}
Burr, S.A.: Generalized ramsey theory for graphs-a survey.
\newblock In: Graphs and Combinatorics: Proceedings of the Capital Conference
  on Graph Theory and Combinatorics at the George Washington University June
  18--22, 1973, pp. 52--75. Springer (1974)

\bibitem{burr1980extremal}
Burr, S.A., Erd{\H{o}}s, P., Faudree, R.J., Rousseau, C., Schelp, R.: An
  extremal problem in generalized ramsey theory.
\newblock Ars Combinatoria \textbf{10}, 193--203 (1980)

\bibitem{christopherson}
Christopherson, B.A.: Stability in common and uncommon places.
\newblock {PhD} dissertation, University of Wyoming (2019)

\bibitem{chvatal1972generalized}
Chv{\'a}tal, V., Harary, F.: Generalized ramsey theory for graphs.
\newblock Bulletin of the American Mathematical Society \textbf{78}(3),
  423--426 (1972)

\bibitem{clarkchevalley}
Clark, P.: The chevalley-warning theorem (featuring... the erdos-ginzburg-ziv
  theorem) \urlprefix\url{http://math.uga.edu/∼
  pete/4400ChevalleyWarning.pdf}

\bibitem{conlon2015erdHos}
Conlon, D., Fox, J., Lee, C., Sudakov, B.: The erd{\H{o}}s--gy{\'a}rf{\'a}s
  problem on generalized ramsey numbers.
\newblock Proceedings of the London Mathematical Society \textbf{110}(1), 1--18
  (2015)

\bibitem{dealgebraic}
De~Loera, J.A., Wesley, W.J.: An algebraic perspective on ramsey numbers.
\newblock Séminaire Lotharingien de Combinatoire \textbf{89B}(6) (2023)

\bibitem{eisenbud1995commutative}
Eisenbud, D., Eisenbud, P.: Commutative Algebra: With a View Toward Algebraic
  Geometry.
\newblock Graduate Texts in Mathematics. Springer (1995).
\newblock \urlprefix\url{https://books.google.com/books?id=Fm\_yPgZBucMC}

\bibitem{godsil2001algebraic}
Godsil, C., Royle, G.: Algebraic Graph Theory.
\newblock Graduate Texts in Mathematics. Springer New York (2001).
\newblock \urlprefix\url{https://books.google.com/books?id=pYfJe-ZVUyAC}

\bibitem{gould2010ramsey}
Gould, M.: Ramsey theory.
\newblock University of Oxford pp. 1--15 (2010)

\bibitem{graham2007some}
Graham, R.: Some of my favorite problems in ramsey theory.
\newblock Integers: Electronic Journal of Combinatorial Number Theory
  \textbf{7}(2), A15 (2007)

\bibitem{graham2008old}
Graham, R.: Old and new problems and results in ramsey theory.
\newblock In: Horizons of Combinatorics, pp. 105--118. Springer Berlin
  Heidelberg (2008)

\bibitem{green2008primes}
Green, B., Tao, T.: The primes contain arbitrarily long arithmetic
  progressions.
\newblock Annals of Mathematics pp. 481--547 (2008)

\bibitem{harary2006recent}
Harary, F.: Recent results on generalized ramsey theory for graphs.
\newblock In: Graph Theory and Applications: Proceedings of the Conference at
  Western Michigan University, May 10--13, 1972 Sponsored jointly by Western
  Michigan University and the US Army Research Office-Durham, under Grant
  Number DA-ARO-D-31-124-72-G155, pp. 125--138. Springer (2006)

\bibitem{harris2013algebraic}
Harris, J.: Algebraic Geometry: A First Course.
\newblock Graduate Texts in Mathematics. Springer New York (2013).
\newblock \urlprefix\url{https://books.google.com/books?id=U-UlBQAAQBAJ}

\bibitem{isaacs2009algebra}
Isaacs, I.: Algebra: A Graduate Course.
\newblock Graduate Studies in Mathematics. American Mathematical Society
  (2009).
\newblock \urlprefix\url{https://books.google.com/books?id=5tKq0kbHuc4C}

\bibitem{kemper2010course}
Kemper, G.: A Course in Commutative Algebra.
\newblock Graduate Texts in Mathematics. Springer Berlin Heidelberg (2010).
\newblock \urlprefix\url{https://books.google.com/books?id=8kxlj48DWM4C}

\bibitem{tattersall2005elementary}
Tattersall, J.J.: Elementary number theory in nine chapters.
\newblock Cambridge University Press (2005)

\bibitem{zhang2014bounded}
Zhang, Y.: Bounded gaps between primes.
\newblock Annals of Mathematics pp. 1121--1174 (2014)

\end{thebibliography}

%
\section*{Statements and Declarations}

\noindent \textbf{Conflict of interest statement}: The authors declare that they have no conflict of interest.

\noindent \textbf{Competing interests statement}: The authors have no relevant financial or non-financial interests to disclose.

\noindent \textbf{Funding statement}: The authors declare that no funds, grants, or other support were received during the preparation of this manuscript.

\noindent \textbf{Data availability statement}: This manuscript has no associated data.

\end{document}